\documentclass[11pt, a4paper, oneside]{article}
\usepackage{amsmath, amssymb, amsthm, amsfonts, amsxtra, latexsym, amscd,
pb-diagram,  graphics, hyperref, setspace}
\usepackage [all]{xy}
\theoremstyle{plain}

\newcommand{\ri}{\rightarrow }

\newcommand{\Fm}{\widetilde{F}}

\DeclareMathOperator{\Coker}{Coker}\DeclareMathOperator{\Dis}{Dis}
 \DeclareMathOperator{\Hom}{Hom}
 \DeclareMathOperator{\Aut}{Aut}
\DeclareMathOperator{\Ext}{Ext} \DeclareMathOperator{\Ker}{Ker}
\DeclareMathOperator{\Ob}{Ob}
\DeclareMathOperator{\Obs}{Obs}\DeclareMathOperator{\Out}{Out}

\newtheorem{thm}{\bf Theorem}
\newtheorem{lem}[thm]{\bf Lemma} 
\newtheorem{pro}[thm]{\bf Proposition} 
\newtheorem{hq}[thm]{\bf Corollary} 

\theoremstyle{definition}

\begin{document}


\centerline{\Large\bf
Equivariant Crossed Modules and }
 \centerline{\Large\bf Cohomology of
Groups with Operators}
 \vspace{0.5cm}
\centerline{\textsc{Nguyen Tien Quang$^{1}$ and Pham Thi
Cuc$^{2,*}$}}
 \vspace{0.5cm}
  \noindent\textit{ $^1$Department of Mathematics, Hanoi National University of
Education, Hanoi, Vietnam}

\noindent \textit{$^2$Natural Science Department, Hongduc
University, Thanhhoa, Vietnam}

\renewcommand{\thefootnote}{}

\footnote{$^*$ Corresponding author.  \textit{Email adrdresses:}
cn.nguyenquang@gmail.com (N. T. Quang), cucphamhd@gmail.com (P. T.
Cuc)}

 \vspace{0.5cm}
\begin{abstract}
In this paper we study equivariant crossed modules in its link with
strict graded categorical groups. The resulting Schreier theory for
equivariant group extensions of the type of an equivariant crossed
module generalizes both the theory of group extensions of the type
of a crossed module and the one of equivariant group extensions.
\end{abstract}


\section{Introduction}

Crossed modules and categorical groups have been used widely and
independently, and in various contexts. Later,  Brown and Spencer
\cite{Br76} show   that crossed modules are defined by $\mathcal
G$-groupoids, and hence crossed modules can be studied by means of
category theory. The notion of $\mathcal G$-groupoid  is also called
{\it strict 2-group} by Baez and Lauda \cite{Baez}, or {\it strict
categorical group } by Joyal and Street \cite{J-S}.

A {\it  categorical group}  is a monoidal category in which every
morphism is invertible and every object has a weak inverse. (Here, a
weak inverse of an object $x$ is an object $y$ such that $x\otimes
y$ and $y\otimes x$ are both isomorphic to the unit object.) A {\it
strict}  categorical group is a strict monoidal category in which
every morphism is invertible and every object has a strict inverse
(so that $x\otimes y$ and $y\otimes x$ are actually equal to the
unit object).


 Graded categorical groups were
originally introduced by  Fr\"{o}hlich and Wall in \cite{Fro}.
  Cegarra et al  \cite{C2002} have proved a precise theorem on the homotopy
 classification of graded categorical groups and their
 homomorphisms thanks to the 3-dimensional equivariant cohomology group
 in the sense of
\cite{C20022}. These results were applied then to give an
appropriate treatment of the equivariant group extensions with a
non-abelian kernel in \cite{C2002}.


 Brown and Mucuk  classified group extensions of the type of a
  crossed module in \cite{Br94}. An another generalized version of group extension
 is the equivariant group extension stated by   Cegarra et al thanks
 to the graded categorical group theory
 (see \cite{C2002}). One can recognize a generalization of both theories
 by means of $\Gamma$-crossed modules and strict graded categorical groups, which we deal with in this work.

The plan of this paper, briefly, is as follows.
 After this introductory Section 1,  Section 2 is devoted to
 recalling some fundamental results and notions of  reduced
 graded categorical groups, the obstruction theory of a monoidal $\Gamma$-functor and a result on
 factor sets that will be used in the next section.
    In Section 3 we introduce the notion of {\it strict graded categorical group}, and show that
    any $\Gamma$-equivariant  crossed module is defined by a strict graded categorical group.
Then, we prove that the category $\bf{_{\Gamma}Grstr}$ of strict
$\Gamma$-graded categorical groups and regular $\Gamma$-graded
monoidal functors is equivalent to the category $\bf
{_{\Gamma}Cross}$  of $\Gamma$-equivariant crossed modules (Theorem
\ref{dlpl}). A morphism in the category $\bf {_{\Gamma}Cross}$
consists of a homomorphism of $\Gamma$-equivariant crossed modules,
$(f_1,f_0):\mathcal M\ri \mathcal M'$, and an element of the group
of equivariant 2-cocycles $Z^2(\pi_0\mathcal M, \pi_1\mathcal M').$
 This result contains a classical one,  Theorem 1 \cite{Br76}.

Last, Section 4 is dedicated to stating  Schreier theory for
equivariant group extensions of the type of a $\Gamma$-crossed
module by means of $\Gamma$-graded monoidal functors (Theorem
\ref{schr}, Theorem \ref{dlc}), then the classification theorem of
group extensions of the type of a crossed module of  Brown and Mucuk
   (Theorem 5.2 \cite {Br94}) and that of $\Gamma$-group extensions of   Cegarra et al (Theorem 4.1 \cite{C2002})
   are obtained as particular cases.

\section{ Reduced graded categorical groups}

Throughout  $\Gamma$ is a fixed group. Let us recall that a  {\it
$\Gamma$-group} $\Pi$ means a group $\Pi$ enriched with a left
$\Gamma$-action by automorphisms, and that an {\it (left)
$\Gamma$-equivariant $\Pi$-module} is a  $\Gamma$-module $A$, that
is, an abelian   $\Gamma$-group, endowed with a   $\Pi$-module
structure such that $\sigma(xa)=(\sigma x)(\sigma a)$ for all
$\sigma\in\Gamma$, $x\in \Pi$ and $a\in A$. A {\it
$\Gamma$-homomorphism} $f:\Pi\ri \Pi'$ of $\Gamma$-groups is a group
homomorphism satisfying $f(\sigma x)=\sigma f(x)$,
$\sigma\in\Gamma$, $x\in\Pi$.

 We regard  the group $\Gamma$ as a category with one object, say
 $\ast$, where the morphisms are elements of $\Gamma$ and the
 composition is the group operation. A category $\mathbb P$ is
 {\it $\Gamma$-graded} if there is a functor $gr:\mathbb P\ri\Gamma$.
  The grading $gr$ is said to be {\it stable} if for any
 object $X\in \mathrm{Ob}\mathbb P$ and any $\sigma\in \Gamma$
 there exists an isomorphism $f$ in $\mathbb P$ with domain $X$ and
$gr(f)=\sigma$. A {\it  $\Gamma$-graded   monoidal category}
$\mathbb P=(\mathbb
 G,gr,\otimes, I, \mathbf{a},\mathbf{l},\mathbf{r})$
 consists of:\\
 \indent 1.  a stable $\Gamma$-graded category $(\mathbb P,gr)$,
  $\Gamma$-graded functors
$\otimes:\mathbb P\times_\Gamma\mathbb P\rightarrow\mathbb P$ and
$I:\Gamma\rightarrow\mathbb P$,\\
\indent 2.  natural isomorphisms of grade 1 ${\bf a}_{X,Y,Z}:
(X\otimes Y)\otimes Z \stackrel{\sim}{\rightarrow}
 X\otimes (Y\otimes Z ), {\bf l}_X: I\otimes X \stackrel{\sim}{\rightarrow} X,
  {\bf r}_X: X\otimes I \stackrel{\sim}{\rightarrow} X$ such that,
  for all $X,Y,Z,T\in \Ob\mathbb P$, the following two coherence
  conditions hold:
\begin{equation*}
{\bf a}_{X,Y,Z\otimes T}{\bf a}_{X\otimes Y, Z,T}= (id_X \otimes
{\bf a}_{Y,Z,T}) {\bf a}_{X,Y\otimes Z,T}({\bf a}_{X,Y,Z}\otimes
id_T), \label{2.1}\end{equation*}  \begin{equation*}(id_X \otimes
{\bf l}_Y){\bf a}_{X,I,Y}= {\bf r}_X \otimes id_Y. \label{2.2}
\end{equation*}

A {\it graded categorical group} is a graded monoidal category
$\mathbb P$ in which every object is invertible and every morphism
is an isomorphism. In this case, the subcategory
  Ker$\mathbb P$ consisting of all objects of $\mathbb P$
and all morphisms of grade 1 in $\mathbb P$ is a categorical group.

If $\mathbb P,\mathbb P'$ are  $\Gamma$-monoidal categories, then a
{\it  graded monoidal functor}
 $(F,\widetilde{F},$ $ F_\ast):\mathbb P\ri\mathbb P'$ consists of a  $\Gamma$-graded functor
 $F : \mathbb P \rightarrow \mathbb P'$, natural isomorphisms of grade 1
$\widetilde{F}_{X,Y} :FX \otimes FY \rightarrow F(X\otimes Y),$ and
an isomorphism of grade 1
 $F_\ast : I' \rightarrow FI$,
such that,
  for all $X,Y,Z\in\Ob\mathbb P$, the following  coherence
  conditions hold:
\begin{equation*} \widetilde{F}_{X,Y\otimes Z}(id_{FX}\otimes
\widetilde{F}_{Y,Z}) {\bf a}_{FX,FY,FZ}= F({\bf
a}_{X,Y,Z})\widetilde{F}_{X\otimes Y,Z}(\widetilde{F}_{X,Y}\otimes
id_{FZ}), \label{6.5}
\end{equation*}
 \begin{equation*}F({\bf r}_X)\widetilde{F}_{X,I}(id_{FX}\otimes
F_\ast)= {\bf r}_{FX},\
 F({\bf l}_X)\widetilde{F}_{I,X}(F_\ast\otimes id_{FX}) = {\bf l}_{FX}. \label{6.6}
\end{equation*}
Let  $(F,\widetilde{F},F_\ast), (F',\widetilde{F}',F'_\ast)$ be two
$\Gamma$-graded monoidal functors. A {\it graded   monoidal natural
equivalence}   $\theta:F\stackrel{\sim}{\rightarrow} F'$ is a
natural equivalence  of functors such that  all isomorphisms
$\theta_X:FX\rightarrow F'X$ are of grade  1, and for all $X,Y\in
\Ob\mathbb P$, the following  coherence
  conditions hold:
\begin{align}\label{3.4}
\widetilde{F}'_{X,Y}(\theta_X\otimes\theta_Y)=\theta_{X\otimes
Y}\widetilde{F}_{X,Y},\;\theta_I F_\ast=F'_\ast.
\end{align}

The authors of \cite{C2002} showed that any  $\Gamma$-graded
categorical group $\mathbb P=(\mathbb
 G,gr,\otimes,$ $ I, \mathbf{a},\mathbf{l},\mathbf{r})$ determines a triple
$(\Pi,A, h)$, where\\
\indent 1. the set $\Pi=\pi_0\mathbb P$ of 1-isomorphism classes of the objects in $\mathbb P$ is a $\Gamma$-group,\\
 \indent 2. the set $A=\pi_1\mathbb P$ of 1-automorphisms of the unit object $I$ is a
$\Gamma$-equivariant $\Pi$-module,\\ 
 \indent 3. the third invariant is an equivariant cohomology class  $h\in
H^3_\Gamma(\Pi,A)$.

Based on these data, they constructed a   $\Gamma$-graded
categorical group $\mathcal S_\mathbb P$, denoted by
$\int_\Gamma(\Pi,A,h)$, which is graded monoidally equivalent to
$\mathbb P$. We call $\mathcal S_\mathbb P$ a {\it reduction} of the
$\Gamma$-graded categorical group $\mathbb P$.

Let $\mathbb P$ and $\mathbb P'$ be $\Gamma$-graded  categorial
groups, $\mathcal S_{\mathbb P}=\int_\Gamma(\Pi, A,h)$ and $\mathcal
S_{\mathbb P'}=\int_\Gamma(\Pi',A',h')$ be their reductions,
respectively. A $\Gamma$-functor $F: \mathcal S_{\mathbb P}\ri
\mathcal S_{\mathbb P'}$ is  {\it of type} $(\varphi,f)$ if
\begin{equation*}
F(x)=\varphi(x),\ \ F(a,\sigma)=(f(a), \sigma),\;\;x\in\Pi,\;a\in
A,\;\sigma\in\Gamma,
\end{equation*}
where $\varphi:\Pi\ri\Pi'$ is a $\Gamma$-homomorphism (so that $A'$
becomes a  $\Gamma$-equivariant $\Pi$-module via $\varphi$) and $f$
is a  homomorphism of $\Gamma$-equivariant $\Pi$-modules (that is, a
homomorphism which is both of $\Gamma$- and   $\Pi$-modules). In
this case, we call  $(\varphi,f)$   a {\it pair of
$\Gamma$-homomorphisms} and
 \begin{equation}\label{ct}
 \xi=\varphi^{\ast}h'-f_{\ast}h
 \end{equation}
  an {\it obstruction} of the $\Gamma$-functor $F$.



\begin{pro}[Theorem 3.2 \cite{C2002}\label{dl2.2a}]  Let $\mathbb P$ and $\mathbb P'$ be $\Gamma$-graded
 categorial groups, $\mathcal S_{\mathbb P}=\int_\Gamma(\Pi, A,h)$ and $\mathcal
S_{\mathbb P'}=\int_\Gamma(\Pi',A',h')$ be their reductions,
respectively.

 $\mathrm{i)}$ Every $\Gamma$-graded monoidal functor $(F,\widetilde{F}):\mathbb P\ri\mathbb P'$
 induces  one $S_F:\mathcal S_{\mathbb P}\ri \mathcal S_{\mathbb P'}$ of type $(\varphi,f)$.

 $\mathrm{ii)}$  Every $\Gamma$-graded monoidal functor $\mathcal S_{\mathbb P}\ri \mathcal S_{\mathbb P'}$
  is a $\Gamma$-functor of type $(\varphi,f)$.

$\mathrm{iii)}$   A $\Gamma$-graded   functor $F:\mathcal S_{\mathbb
G}\ri\mathcal S_{\mathbb P'}$ of type $(\varphi, f)$ is realizable,
that is, it induces a $\Gamma$-graded monoidal functor, if and only
if its obstruction  $\overline{\xi}$ vanishes in
   $H^3_\Gamma(\Pi, A')$. Then, there is a bijection
 \begin{equation*}
     {\mathrm{Hom}}_{(\varphi, f)}[\mathbb P, \mathbb P']\leftrightarrow H^2_\Gamma(\Pi,
 A'),
\end{equation*}
where $\Hom_{(\varphi, f)}[\mathbb P, \mathbb P']$ is the set of all
homotopy classes of monoidal
 $\Gamma$-functors from $\mathbb P$ to $ \mathbb P'$ inducing the
 pair of $\Gamma$-homomorphisms $(\varphi, f)$.
\end{pro}

This result is  stated in Theorem 3.2 \cite{C2002} by Cegarra et al
by means of the notions of  $\Gamma$-pairs and of homomorphism of
$\Gamma$-pairs. It is also deduced from Propositions 4, 5, and
Theorem 6  \cite{QCT} with some appropriate modifications.

$\bullet$ {\it Definition of a factor set with coefficients in a
categorical group.}

The notion of factor set in the Schreier-Eilenberg-Mac Lane theory
for group extensions  has been raised to
 categorical level by Grothendieck \cite{Gro} and also applied in  \cite{U0}, \cite{C2001}, \cite
 {Q2012}. In this paper  we use this notion to define a {\it strict}
 graded categorical group.




\vspace{0.2cm} \noindent {\bf Definition.}  A {\it factor set}
$\mathcal F$  on $\Gamma$ with coefficients in a
  categorical group $\mathbb G$ (or a {\it pseudo-functor} from $\Gamma$ to the category of categorical groups
  in the sense of Grothendieck \cite{Gro}) consists of a family of    monoidal
autoequivalences $F^{\sigma}:\mathbb G\rightarrow \mathbb G,
\sigma\in\Gamma, $ and isomorphisms between   monoidal functors
$\theta^{\sigma,\tau} : F^\sigma F^\tau \rightarrow F^{\sigma\tau}
$, $ \sigma,\tau\in\Gamma$,
satisfying the conditions:\\
\indent i) $F^{1} = id_\mathbb G$,\\
\indent ii) $ \theta^{1,\sigma} = id_{F^\sigma} = \theta^{\sigma,1}$, $\sigma \in\Gamma$,\\
\indent iii) $\theta^{\sigma\tau,\gamma}\circ \theta^{\sigma,\tau}
F^{\gamma} =\theta^{\sigma,\tau\gamma}\circ
F^{\sigma}\theta^{\tau,\gamma}$, for all $
\sigma,\tau,\gamma\in\Gamma$.

We write $\mathcal F=(\mathbb G,F^{\sigma},\theta^{\sigma,\tau})$,
or simply $(F,\theta)$.

The following lemma comes from an analogous result on graded
monoidal categories  \cite{C2001} or a part of Theorem 1.2 \cite{U}.
 We sketch the proof since we need some of its details.

\begin{lem}\label{b01} Each $\Gamma$-graded categorical group
$(\mathbb{P},gr)$ determines a factor set $\mathcal F $ on $\Gamma$
with coefficients in the categorical group $\Ker\mathbb P$.
\end{lem}
\begin{proof}

For  $\sigma\in\Gamma$, we define a monoidal autoequivalence
$F^{\sigma} = (F^{\sigma},
    \widetilde{F}^{\sigma}):$
$\Ker\mathbb{P}\rightarrow \Ker\mathbb{P}$ as follows: for each
$X\in\Ker \mathbb P$, since the grading $gr$ is stable, there exists
an isomorphism
    $\Upsilon ^{\sigma}_{X} : X \stackrel\sim\rightarrow
  F^{\sigma}X$, where $F^{\sigma}X\in $ Ker$\mathbb{P}$, and
   $gr(\Upsilon ^{\sigma}_{X})=\sigma$. In particular, when  $\sigma = 1$
    we take $F^{1}X = X$ and $\Upsilon^{1}_{X} = id_{X}$.
A morphism $f:X\rightarrow Y$ of grade 1 in Ker$\mathbb{P}$ is
carried to
  the unique morphism $F^{\sigma}(f)$ in Ker$\mathbb{P}$
determined by
$$F^{\sigma}(f)=\Upsilon^{\sigma}_{Y}\circ f\circ(\Upsilon^{\sigma}_{X})^{-1}.$$
 The natural isomorphism  $\widetilde{F}^{\sigma}_{X,Y} :
{F^{\sigma}X} \otimes {F^{\sigma}Y} \stackrel\sim\longrightarrow
F^{\sigma}(X\otimes Y)$ is   determined by
$$\widetilde{F}^{\sigma}_{X,Y}=(\Upsilon^{\sigma}_{X}
\otimes \Upsilon^{\sigma}_{Y})\circ(\Upsilon^{\sigma}_{X\otimes
Y})^{-1}.$$
 Furthermore, for each pair $\sigma,\tau\in\Gamma$ there
is an isomorphism of monoidal functors  $ \theta^{\sigma,\tau}
:F^{\sigma}F^{\tau} \stackrel\sim\longrightarrow F^{\sigma\tau}$,
with $\theta^{1,\sigma} =id_{F^{\sigma}} =\theta^{\sigma,1}$,  which
is defined, for any $X \in \Ob\mathbb P$, by 
$$\theta^{\sigma,\tau}_{X}=\Upsilon^{\sigma}_{F^{\tau}X}\circ
\Upsilon^{\tau}_{X}\circ(\Upsilon^{\sigma\tau}_{X})^{-1}.$$ The pair
$(F,\theta)$ constructed as above is a factor set.
\end{proof}

\section{ Strict graded  categorial groups  and  $\Gamma$-crossed modules}

The objective of this paper is to extend the results on crossed
modules and on equivariant group extensions. The notion of
$\Gamma$-crossed module is a generalization of that of crossed
module of groups of Whitehead \cite{White49}. First,  observe that
if $B$ is a $\Gamma$-group, then the group Aut$B$ of automorphisms
of $B$ is also a $\Gamma$-group with the action
$$(\sigma f)(b)=\sigma(f(\sigma^{-1}b)),\;b\in B,\;f\in \Aut B.$$
Then, the homomorphism $\mu:B\ri \Aut B, b\mapsto \mu_b$ ($\mu_b$ is
the inner automorphism of $B$ given by conjugation with   $b$) is a
homomorphism of $\Gamma$-groups. Indeed, for all $\sigma\in\Gamma,
a,b\in B,$ one has
$$\mu_{\sigma b}(a)=\sigma
b+a-\sigma
b=\sigma(b+\sigma^{-1}a-b)=\sigma(\mu_b(\sigma^{-1}a))=(\sigma\mu_b)(a).$$

\vspace{0.2cm} \noindent {\bf Definition.} Let $B,D$ be
$\Gamma$-groups.  A {\it $\Gamma$-crossed module} is a quadruple
$(B,D,d,\vartheta)$, where $d:B\ri D,\;\vartheta:D\ri$ Aut$B$ are
$\Gamma$-homomorphisms satisfying the following conditions:
\\
\indent $C_1.\ \vartheta d=\mu,$\\
  \indent $C_2.\ \label{ct4b}d(\vartheta_x(b))=\mu_x(d(b)),$\\
\indent $C_3.\ \label{ct4c}\sigma(\vartheta_x(b))=\vartheta_{\sigma x}(\sigma b),$\\
where $\sigma\in\Gamma,x\in D, b\in B,$ $\mu_x$  is the inner
automorphism given by conjugation with $x$.

A $\Gamma$-crossed module is also called an equivariant crossed
module by Noohi \cite{N}.

 \vspace{0.2cm} \noindent {\bf Examples.} Standard
examples of $\Gamma$-crossed modules are:

1.  $(B,D,i,\mu)$, where  $i:B\ri D$ is an   inclusion
$\Gamma$-homomorphism of  a normal subgroup.

2.  $(B,D,{\bf 0},\vartheta)$, where $B$ is a $D$-module,   ${\bf
0}:B\ri D$ is the zero $\Gamma$-homomorphism, and $\vartheta$ is the
module action.

3. $(B,\Aut B,\mu,0)$, where  $\mu:B\ri \Aut B$ is the
$\Gamma$-homomorphism of any $\Gamma$-group $B$ which is given by
conjugation.

4.  $(B,D,p,\vartheta)$, where $p:B\ri D$ is  a $\Gamma$-surjective
such that $\Ker p\subset Z(B)$, $\vartheta$ is  given by
conjugation.

{\it Note on notations.} For convenience,  we denote by the addition
for the operation in $B$ and by the multiplication for that in $D$.
In this paper the $\Gamma$-crossed module $ (B,D,d,\vartheta)$ is
sometimes denoted by $B\stackrel{d}{\rightarrow}D$, or $B\rightarrow
D$. In this section notations $\mathcal M,\ \mathcal M'$ refer to
the $\Gamma$-crossed modules $(B,D,d,\vartheta)$,
$(B',D',d',\vartheta')$, respectively.

The following properties follow from the definition of
$\Gamma$-crossed module.
\begin{pro} \label{md2}
Let   $\mathcal M$ be a $\Gamma$-crossed module. \\
 \indent $\mathrm{i)}\;\mathrm{Ker}d$ is a $\Gamma$-subgroup in $Z(B)$.\\
 \indent $\mathrm{ii)} \;\mathrm{Im}d$ is both a normal subgroup in  $D$
and  a  $\Gamma$-group.\\
 \indent $\mathrm{iii)}$  The $\Gamma$-homomorphism $\vartheta$ induces
 one $\varphi:D\ri \mathrm{Aut}(\mathrm{Ker}d)$  by
$$\varphi_x=\vartheta_x|_{\mathrm{Ker}d}.$$
 \indent $\mathrm{iv)}\; \mathrm{Ker}d$ is a left $\Gamma$-equivariant
$\mathrm{Coker}d$-module under the actions
$$sa=\varphi_x(a),\ \sigma s=[\sigma x],\ a\in\mathrm{ Ker}d, \ x\in s\in \mathrm{Coker}d.$$
\end{pro}
The groups $\mathrm{Ker}d$ and $\mathrm{Coker}d$ are denoted by
$\pi_1\mathcal M$ and $\pi_0\mathcal M$, respectively.

It is well known that each crossed module of groups can be seen as a
strict categorical group  (see  \cite{Br76}, \cite{J-S} Remark 3.1).
Crossed modules of groups can be {\it enriched} in some ways to
become, for example, {\it crossed bimodules over rings}, or {\it
equivariant crossed modules}. In the former case, each crossed
bimodule can be seen as a strict Ann-category \cite{QC}. In the
later case, we shall show that each crossed module of
$\Gamma$-groups can be identified with a {\it strict $\Gamma$-graded
categorical group}.  We now state this definition.

Recall that 
 if $(F, \Fm, F_\ast)$ is a monoidal functor between
categorical groups, then the isomorphism  $ F_\ast:I'\ri FI$ can be
deduced from
 $F$ and $\Fm$, so we can omit $F_*$ when not necessary.
 A monoidal functor  $(F, \Fm):\mathbb
G\ri\mathbb G'$ between two categorical groups is termed  {\it
regular} if
$$F(x)\otimes F(y)=F(x\otimes y), \ F(b)\otimes F(c)=F(b\otimes c),$$
for all $x,y\in $ Ob$\mathbb G$, $b,c\in$ Mor$\mathbb G$.
 A factor set  $(F,\theta)$ on $\Gamma $ with coefficients in a
categorical group   $\mathbb G$  is \emph{regular} if
$\theta^{\sigma,\tau} =id$ and $F^\sigma $ is a regular monoidal
functor, for all $\sigma \in \Gamma $.

\vspace{0.2cm}
\noindent {\bf Definition.} A graded categorical group $(\mathbb P,gr)$ is said to be {\it strict} if\\
\indent $\mathrm{i)}$ $\Ker \mathbb P$ is a strict categorical group,\\
\indent $\mathrm{ii)}$ $\mathbb P$ induces a regular factor set
$(F,\theta)$ on $\Gamma $ with coefficients in a categorical group
$\Ker \mathbb P$.

Equivalently, a graded categorical group $(\mathbb P,gr)$ is
\emph{strict} if it is a $\Gamma $-graded extension of a strict
categorical group by a regular factor set.

\vspace{0.2cm}
 $\bullet$ Construction of the strict $\Gamma$-graded categorical group
$\mathbb P_{\mathcal M}:=\mathbb P$ associated to the
$\Gamma$-crossed module $\mathcal M$.

Objects of $\mathbb P$ are elements of the group $D$, a
$\sigma$-morphism $x\ri y$ is a pair $(b,\sigma)$, where $b\in
B,\sigma\in\Gamma$ such that $\sigma x=d(b)y$.  The composition of
two morphisms is defined by
 \begin{equation}
\label{lh1}(x\stackrel{(b,\sigma)}{\ri}y\stackrel{(c,\tau)}{\ri}z)=(x\xrightarrow{(\tau b+c, \tau\sigma)}z).
\end{equation}This composition is associative and unitary since $B$ is a
$\Gamma$-group.

For any morphism $(b,\sigma)$ in $\mathbb P$, one has
$$(b,\sigma)^{-1}=(-\sigma^{-1}b,\sigma^{-1}),$$
so that $\mathbb P$ is a groupoid.

The tensor operation on objects is given by the multiplication in
the group $D$, and for two morphisms
$(x\stackrel{(b,\sigma)}{\ri}y), (x'\stackrel{(c,\sigma)}{\ri}y'),$
then
\begin{equation}\label{lh2}
(x\stackrel{(b,\sigma)}{\ri}y)\otimes(x'\xrightarrow{(c,\sigma)}y')=(xx'\xrightarrow{(b+\vartheta_yc,\sigma)}yy'),
\end{equation}
which is a functor thanks to the compatibility of the action
$\vartheta$ with the  $\Gamma$-action and conditions  of the
definition of    $\Gamma$-crossed module, as below.

For morphisms $(x\stackrel{(b,\sigma)}{\ri}y
\stackrel{(c,\tau)}{\ri}x), (x'\stackrel{(b',\sigma)}{\ri}y'
\stackrel{(c',\tau)}{\ri}z')$ in $\mathbb P$,
\[\begin{aligned}(x\stackrel{(b,\sigma)}{\ri}y
\xrightarrow{(c,\tau)} x)\otimes (x'\xrightarrow{(b',\sigma)} y'
\xrightarrow{(c',\tau)}
z')&\stackrel{(\ref{lh1})}{=}(x\xrightarrow{(\tau b+c,\tau\sigma)}
z)\otimes(x' \xrightarrow{(\tau b'+c',\tau\sigma)} z')\\
 \;&\stackrel{(\ref{lh2})}{=}(xx'\xrightarrow{(\tau b+c+\vartheta_z(\tau b'+c'),\tau\sigma)}
zz'),\end{aligned}\]
\[\begin{aligned} {[(x\xrightarrow{(b,\sigma)}y)\otimes(x'\xrightarrow{(b',\sigma)}y')]\circ} &
[(y\xrightarrow{(c,\tau)} z)\otimes(y'\xrightarrow{(c',\tau)}z')]\\
 \stackrel{(\ref{lh2})}{=}&(xx'\xrightarrow{(b+\vartheta_yb',\sigma)}yy')\circ(yy'\xrightarrow{(c+\vartheta_zc',\tau)}zz')\\
\stackrel{(\ref{lh1})}{=}&(xx'\xrightarrow{(\tau(b+\vartheta_yb')+c+\vartheta_zc',\tau\sigma)}zz').\end{aligned}\]
The fact that $\otimes$ is a functor is equivalent to
$$\tau
b+c+\vartheta_z(\tau b'+c')=\tau(b+\vartheta_yb')+c+\vartheta_zc'.$$
This follows from
 \[\begin{aligned}
\tau(\vartheta_yb') \stackrel{(C_3)}{=} \vartheta_{\tau y}(\tau b')
 =\vartheta_{(dc) z}(\tau b')
 \stackrel{(C_1)}{=}\mu_c(\vartheta_{ z}(\tau b')).\end{aligned}\]

The associativity and unit constraints with respect to  tensor
product are strict. 
The graded functor    is defined by $gr (b,\sigma)= \sigma, $ and
the unit graded functor $I: \Gamma \rightarrow \mathbb P$ by
\begin{align*} I(*\xrightarrow{\sigma } *)=(1\xrightarrow{(0,\sigma) } 1). \end{align*}
Since Ob$\mathbb P=D$ is a group and $x\otimes y=xy$, every object
of   $\mathbb P$ is invertible, whence $\mathrm{Ker}\mathbb P$ is a
strict categorical group.

We next show that $\mathbb P$ induces a regular factor set
$(F,\theta)$ on
 $\Gamma$ with coefficients in Ker$\mathbb P$.
For each $x\in D,\sigma\in\Gamma$, we set $F^\sigma(x)=\sigma x$,
  $\Upsilon ^\sigma_x =(x\xrightarrow{(0,\sigma)}\sigma x)$. Then, according to the proof of Lemma \ref{b01},
we have $F^\sigma(b,1)=(\sigma b,1)$ and $\theta^{\sigma,\tau}=id$.
Now, it follows from the $\Gamma$-crossed module structure of $B\ri
D$ that $F^\sigma$ is a regular monoidal functor.

Thus, $\mathbb P$ is a strict $\Gamma$-graded categorical group.

$\bullet$ Construction of the $\Gamma$-crossed module {\it
associated } the strict $\Gamma$-graded categorical group  $\mathbb
P$.

Set
\begin{equation} D= \Ob \mathbb P ,\;\;B=\{x\xrightarrow{b}1\ |\ x\in D,\  gr(b)=1 \}. \notag \end{equation}
The operations in  $D$ and in $B$ are given by
 $$xy=x\otimes y,\;\;b+c=b\otimes c, $$
respectively. Then, $D$ becomes a group in which the unit is $1$,
the inverse of $x$ is $x^{-1}$ ($x\otimes x^{-1}=1$). $B$ is a group
in which the zero is the morphism   $(1\xrightarrow{id_1} 1)$ and
the negative of $(x\xrightarrow{b} 1)$ is the morphism
 $(x^{-1}\xrightarrow{\overline{b}}  1 )(b\otimes
\overline{b}=id_1)$.

By the definition of $\mathbb P$, its kernel Ker$\mathbb P$ is a
strict categorical group and $\mathbb P$ has a regular factor set
$(F,\theta)$. Thus, $D,B$ are $\Gamma $-groups in which the actions
are respectively defined by
$$\sigma x= F^\sigma (x),\; x\in D, \sigma \in \Gamma, $$
$$\sigma b=F^\sigma (b),\ b\in B.$$
The correspondences $d:B\rightarrow D$ and $\vartheta: D\rightarrow
\Aut B$ are given by
 $$d(x\xrightarrow{b}  1)=x,$$
$$\vartheta _y(x\xrightarrow{b}1)= (yxy^{-1}\xrightarrow{id_y + b + id_{y^{-1}}}
1),$$ respectively. Since  $B,D$ are $\Gamma$-groups, it is easy to
see that $d,\vartheta$  are $\Gamma$-homomorphisms.

\vspace{0.2cm} \noindent{\bf Definition.} A {\it homomorphism}
$(f_1,f_0):\mathcal M\ri \mathcal M'$ of $\Gamma$-crossed modules
consists of
  $\Gamma$-homomorphisms $f_1:B\ri B'$,
$f_0:D\ri D'$ satisfying\\
\indent $H_1.\ \label{gr1} f_0d=d'f_1$,\\
\indent $H_2.\ \label{gr2}
f_1(\vartheta_xb)=\vartheta'_{f_0(x)}f_1(b),$\\
  for all $x\in D,b\in B$.

  The following lemmas state the relation between homomorphisms of
  $\Gamma$-crossed modules and graded monoidal functors between
  corresponding associated graded categorical groups. Observe that a morphism   $(x\xrightarrow{(b,\sigma)}y)$ in
   $\mathbb P_{B\ri D}$ can be written as
 $$x\xrightarrow{(0,\sigma)}\sigma
 x\xrightarrow{(b,1)}y,$$
 and a $\Gamma$-graded monoidal functor  $(F,\widetilde{F}):\mathbb P_{\mathcal M} \ri \mathbb P_{\mathcal M'}$
 defines a function  $f:D^2\cup D\times\Gamma\ri B$ by
 \begin{equation}\label{f}
 (f(x,y),1)=\widetilde{F}_{x,y},\quad (f(x,\sigma),\sigma)=F(x\stackrel{(0,\sigma)}{\rightarrow}\sigma x)
 \end{equation}

\begin{lem}\label{t1}
 Let
$(f_1,f_0):\mathcal M\ri \mathcal M'$ be a homomorphism of
$\Gamma$-crossed modules. Then, there exists a $\Gamma$-graded
monoidal functor
 $(F,\widetilde{F}):\mathbb P_{\mathcal M} \ri \mathbb P_{\mathcal
M'}$  defined by $F(x)=f_0(x),\;F(b,1)=(f_1(b),1),$ if and only if
$f=p^\ast\varphi,$ where $\varphi\in Z^2_\Gamma(\Coker d, \Ker d')$,
$p:D\ri \Coker d$ is a canonical projection.
  \end{lem}

\begin{proof}
Since $f_0$ is a homomorphism and $Fx=f_0(x)$,
$\widetilde{F}_{x,y}:FxFy\ri F(xy)$ is a morphism of grade  1 in
$\mathbb P'$ if and only $df(x,y)=1'$, or $f(x,y)\in\Ker d'\subset
Z(B')$.

Also, since $f_0$ is a $\Gamma$-homomorphism,
$(Fx\xrightarrow{(f(x,\sigma),\sigma)}F\sigma x)$ is a morphism of
grade $\sigma$ in $\mathbb P'$ if and only if $df(x,\sigma)=1'$, or
$f(x,\sigma)\in\Ker d'\subset Z(B')$. In particular, when $\sigma=1$
then $f(x,1_\Gamma)=f_1(0)=0.$

The fact that $f_1$ is a group homomorphism is equivalent to the
condition of $F$ preserving composition of morphisms of grade  1.
The condition of $F$ preserving the composition of morphisms of form
$(0,\sigma)$ is equivalent to
\begin{equation}\label{1s}\tau f(x,\sigma)+f(\sigma x,\tau)=f(x,\tau\sigma).
\end{equation}

$\bullet$ The condition of $\widetilde{F}_{x,y}$ being  natural
isomorphisms.

- For morphisms of grade 1, consider the diagram
\begin{equation*}\label{bdtn1}
\begin{diagram}
\node{F(x)F(y)}\arrow{e,t}{\widetilde{F}_{x,y}}\arrow{s,l}{F(b,1)\otimes
F(c,1)}\node{F(xy)}\arrow{s,r}{F[(b,1)\otimes (c,1)]}\\
\node{F(x')F(y')}\arrow{e,b}{\widetilde{F}_{x',y'}}\node{F(x'y').}
\end{diagram}
\end{equation*}
Since the homomorphisms $f_0,f_1$ satisfy the condition $H_2$, the
following equation holds:
$$F(b,1)\otimes F(c,1)=F[(b,1)\otimes(c,1)].$$
%

Then, since $f(x,y),f(x',y')\in Z(B')$, the  above diagram commutes
if and only if
$$f(x,y)=f(x',y').$$
So, $\widetilde{F}$ defines a function $\varphi:\Coker^2d\ri\Ker
d',$
$$\varphi(r,s)=f(x,y),\ r=px,s=py$$
where $p:D\ri\Coker d$ is a canonical projection.

- For morphisms of form $(0,\sigma)$, the diagram
\begin{equation*}\label{bdtn2}
\begin{diagram}
\node{F(x)F(y)}\arrow{e,t}{\widetilde{F}_{x,y}}\arrow{s,l}{F(0,\sigma)\otimes
F(0,\sigma)}\node{F(xy)}\arrow{s,r}{F  [(0,\sigma)\otimes (0,\sigma)]}\\
\node{F(\sigma x)F(\sigma y)}\arrow{e,b}{\widetilde{F}_{\sigma
x,\sigma y}}\node{F(\sigma x)(\sigma y)=F\sigma(xy)}
\end{diagram}
\end{equation*}
 commutes if and only if
\begin{equation*} \sigma f(x,y)-f(\sigma x,\sigma
y)=f(x,\sigma)+\vartheta'_{F(\sigma x)}f(y,\sigma)-f(xy,\sigma),
\end{equation*}
or
\begin{equation}\label{2s} \sigma f(x,y)-f(\sigma x,\sigma
y)=f(x,\sigma)+(\sigma x)f(y,\sigma)-f(xy,\sigma).
\end{equation}
$\bullet$ The commutativity of diagram
\begin{equation*}\label{bdtn3}
\begin{diagram}
\node{Fx}\arrow{e,t}{(f(x,\sigma),\sigma)}\arrow{s,l}{F(b,1) }\node{F\sigma x}\arrow{s,r}{F(\sigma b,1)}\\
\node{Fy}\arrow{e,b}{(f(y,\sigma),\sigma)}\node{F\sigma y}
\end{diagram}
\end{equation*}
leads to
$$f(x,\sigma)+f_1(\sigma b)=\sigma f_1(b)+f(y,\sigma).$$
Since $f_1$ is a $\Gamma$-homomorphism, it follows that
$f(x,\sigma)=f(y,\sigma).$ This gives a function  $\varphi:\Coker
d\times\Gamma\ri \Ker d'$
$$\varphi(r,\sigma)=f(x,\sigma),\ r=px.$$
Therefore, one obtains a function
$$\varphi:\Coker^2d\cup\Coker d\times\Gamma\ri\Ker d'$$
which is normalized in the sense that
$$\varphi(1,s)=\varphi(r,1)=0=\varphi(r,1_\Gamma).$$
The first two equations hold since  $F(1)=1'$ and
$(F,\widetilde{F})$ is  compatible with the unit constraints. The
last one holds since $f(x,1_\Gamma)=0.$

The compatibility of   $(F,\widetilde{F})$ with the associativity
constraints implies
\begin{equation}\label{3s}\vartheta'_{Fx}(f(y,z))+f(x,yz)=f(x,y)+f(xy,z).\end{equation}
It follows from the relations (\ref{1s})-(\ref{3s}) that $\varphi\in
Z^2_\Gamma(\Coker d,\Ker d')$.
\end{proof}

Note that the strict $\Gamma$-graded categorical group  $\mathbb P$
induces a $\Gamma$-action on the group  $D$ of objects and on the
group   $B$ of morphisms of grade 1, we state the following
definition.

\vspace{0.2cm} \noindent{\bf Definition.} A $\Gamma$-graded monoidal
functor $(F,\widetilde{F}):\mathbb
P\ri\mathbb P'$ between two strict  $\Gamma$-graded categorical groups is called   {\it regular } if:\\
\indent $S_1.\ F(x\otimes y)=F(x)\otimes F(y),$\\
\indent $S_2.\ F(\sigma x)=\sigma F(x)$,\\
\indent $S_3.\ F(\sigma b)=\sigma F(b)$,\\
\indent $S_4.\ F( b)\otimes F( c)=F( b\otimes c),$\\
for $x, y\in \Ob\mathbb P,$ and $b,c$ are morphisms of grade  1 in
$\mathbb P$.

\vspace{0.2cm}

The $\Gamma$-graded monoidal functor mentioned in Lemma \ref{t1} is
 regular.

Thanks to Lemma \ref{t1}, one can define the category ${\bf _\Gamma
Cross}$ whose objects are $\Gamma $-crossed modules and whose
morphisms are triples $(f_1, f_0, \varphi)$, where $(f_1, f_0):
\mathcal M\rightarrow
  \mathcal M'$ is a homomorphism of  $\Gamma $-crossed
modules and $\varphi \in Z^2_{\Gamma}(\Coker d, \Ker d')$. The
composition with the morphism   $(f'_1, f'_0, \varphi'): \mathcal
M'\rightarrow
  \mathcal M''$ is given by
  $$(f_1',f_0',\varphi')\circ(f_1,f_0,\varphi)=(f_1'f_1,f_0'f_0, (f'_1)_\ast(\varphi)+f_0^\ast(\varphi')).$$

\begin{lem}\label{n1}
 Let $\mathbb P$ and
$\mathbb P'$ be corresponding strict $\Gamma $-graded categorical
groups associated to   $\Gamma$-crossed modules $\mathcal M$ and
$\mathcal M'$, and let $(F,\widetilde{F}):\mathbb P \ri \mathbb P'$
be a regular $\Gamma$-graded monoidal functor. Then, the triple
$(f_1,f_0,\varphi)$, where\\
\indent $\mathrm{i)}\ f_0(x)=F(x), (f_1(b),1)=F(b,1),  \sigma\in\Gamma,b\in B, x, y\in D,$\\
\indent $\mathrm{ii)}\ p^*\varphi=f$, for $f$ is given by \eqref{f},\\
is a morphism in the category  ${\bf _\Gamma Cross}$.
\end{lem}

\begin{proof}
By the condition $S_1$, $f_0$ is a group homomorphism, and by the
condition $S_2$, $f_0$ is a $\Gamma$-homomorphism. Since $F$
preserves the composition of morphisms of grade 1, $f_1$ is a group
homomorphism. Moreover,   $f_1$ is a $\Gamma$-homomorphism thanks to
the condition $S_3$. Each element $b\in B$ can be seen as  a
morphism  $(db\xrightarrow{(b,1)} 1)$ in $\mathbb P$, and hence
$(f_0(db)\xrightarrow{(f_1(b),1)}1')$ is a morphism in $\mathbb P'$,
that means $H_1$ holds:
$$f_0(db)=d'(f_1(b)).$$ By the condition $S_4$ and the fact that $f_1$
is a homomorphism,   $H_2$ is satisfied:
\[
f_1(\vartheta_yc)=\vartheta'_{f_0(y)}f_1(c).\]

Thus, $(f_1,f_0)$ is a homomorphism of $\Gamma$-crossed modules. By
Lemma \ref{t1}, the function $f$ determines a function $\varphi\in
Z^2_{\Gamma}(\Coker d, \Ker d')$, where $f=p^{*}\varphi,$  $ p:D\ri
\Coker d$ is the canonical projection.
Therefore,  $(f_1,f_0,\varphi)$ is a morphism in ${\bf _\Gamma
Cross}$.
\end{proof}

Denote by
 $${\bf _\Gamma Grstr}$$
 the  category of strict  $\Gamma$-graded categorical groups and
regular  $\Gamma$-graded monoidal functors, we have the following
result. 
\begin{thm}[\label{dlpl}Classification Theorem]   There exists an equivalence
\[\begin{matrix}
 \Phi:{\bf _\Gamma Cross}&\ri&{\bf _\Gamma Grstr},\\
(B\ri D)&\mapsto&\mathbb{P}_{B\ri D}\\
(f_1,f_0, \varphi)&\mapsto&(F,\widetilde{F})
\end{matrix}\]
where $F(x)=f_0(x),\ F(b,1)=(f_1(b),1),$ and
$$F(x\stackrel{(0,\sigma)}{\rightarrow}\sigma x)=(\varphi(px,\sigma),\sigma),\
\widetilde{F}_{x,y}=(\varphi(px,py), 1),$$ for $x,y\in D, b\in
B,\sigma\in\Gamma$.
\end{thm}
\begin{proof} Let $\mathbb P,\mathbb P'$ be the   $\Gamma$-graded categorical groups associated to
 $\Gamma$-crossed modules $\mathcal M,\mathcal M'$, respectively. By Lemma \ref{t1},
the correspondence  $(f_1, f_0, \varphi)\mapsto (F,\widetilde{F})$
defines an injection on the homsets,
$$\Phi:\Hom_{{\bf _\Gamma Cross}}(\mathcal M, \mathcal M')\ri \Hom_{{\bf _\Gamma Grstr}}(\mathbb P_{\mathcal M},
\mathbb P_{\mathcal M'}).$$ By Lemma  \ref{n1},
 $\Phi$ is surjective.

If $\mathbb P$ is a strict $\Gamma$-graded  categorical group and
$\mathcal M_{\mathbb P}$ is its associated $\Gamma$-crossed module,
then $\Phi(\mathcal M_{\mathbb P})=\mathbb P$ (rather than an
isomorphism). Thus, $\Phi$ is an equivalence.
\end{proof}

{\it Remark.} Theorem \ref{dlpl} contains Theorem 1 \cite{Br76}.
Indeed, when  $\Gamma={\bf 1}$, the trivial group, one obtains an
equivalence
$$\Phi:{\bf  Cross}\ri{\bf  Grstr}.$$
In the category {\bf  Cross} the objects  are crossed modules, the
morphisms are triples  $(f_1, f_0, \varphi)$, where $(f_1, f_0):
\mathcal M\rightarrow
  \mathcal M'$ is a homomorphism of crossed modules and  $\varphi \in Z^2(\mathrm{Coker}d, \mathrm{Ker}d')$.
In the category {\bf  Grstr} the objects are strict categorical
group, the morphisms are regular monoidal functor.

Then, the category of  $\mathcal G$-groupoids (by Brown and Spencer
\cite{Br76}) is a subcategory of the category   ${\bf Grstr}$ in
which the morphisms consist of   monoidal functors
$(F,\widetilde{F})$ with $\widetilde{F}=id,$ and the category ${\bf
CrossMd}$ of crossed modules is the subcategory of the category
${\bf Cross}$  in which the morphisms consist of triples  $(f_1,
f_0, \varphi)$ with $\varphi=0$. These two categories are equivalent
via  $\Phi$. Thus, we obtain Theorem  1 \cite{Br76}.

\section {Equivariant group extensions,  $\Gamma$-crossed modules and equivariant group cohomology }

In this section we develop a theory  of equivariant group extensions
of the type of a $\Gamma$-crossed module which extends both group
extension theory of the type of a crossed module \cite{Br94,   Ded,
Tay} and equivariant group extension theory \cite{C2002}.

\vspace{0.2cm}
 \noindent{\bf Definition.} Let $ B\xrightarrow{d} D$ be a $\Gamma$-crossed
module and $Q$ be a $\Gamma$-group. An \emph{equivariant group
extension} of $B$ by $Q$ \emph{of type} $ B\xrightarrow{d} D$ is a
diagram of $\Gamma$-homomorphisms
\begin{align*}  \begin{diagram}
\xymatrix{\mathcal E:\;\;\;\;\;\; 0 \ar[r]& B \ar[r]^j \ar@{=}[d] &E  \ar[r]^p \ar[d]^\varepsilon & Q \ar[r]& 1, \\
& B \ar[r]^d & D}
\end{diagram}
\end{align*}
where the top row is exact,   the family $(B,E,j,\vartheta^0)$ is a
$\Gamma$-crossed module in which $\vartheta^0$ is given by
conjugation, and $(id,\varepsilon)$ is a homomorphism of
$\Gamma$-crossed modules.

\vspace{10pt} Two equivariant extensions of $B$ by $Q$ of  type
$B\xrightarrow{d}D$ are said to be \emph{equivalent} if there is a
morphism of exact sequences
\begin{align*}  \begin{diagram}
\xymatrix{ 0 \ar[r]& B \ar[r]^j \ar@{=}[d] &E  \ar[r]^p
\ar[d]^\alpha & Q \ar[r]\ar@{=}[d] & 1,&\;\;\; E
\ar[r]^\varepsilon&D \\
 0 \ar[r]& B \ar[r]^{j'}  &E'  \ar[r]^{p'}   & Q \ar[r]& 1,&\;\;\; E'
\ar[r]^{\varepsilon'}&D}
\end{diagram}
\end{align*}
such that $\varepsilon'\alpha=\varepsilon$. Obviously, $\alpha $ is
a $\Gamma$-isomorphism.

In the diagram
\begin{equation} \label{dccs} \begin{diagram}
\xymatrix{ 0 \ar[r]& B \ar[r]^j \ar@{=}[d] &E  \ar[r]^p \ar[d]^\varepsilon & Q \ar[r]\ar@{.>}[d]^\psi & 1, \\
 & B \ar[r]^d  &D  \ar[r]^q   & \text{Coker}d }
\end{diagram}
\end{equation}
where $q$ is a canonical  $\Gamma$-homomorphism, since the top row
is exact and $ q\circ \varepsilon \circ j = q\circ d =0,$ there  is
a  $\Gamma$-homomorphism $\psi: Q\rightarrow \text{Coker}d $ such
that the right hand side square commutes. Moreover, $\psi$ is
dependent only on the equivalence class of the extension $\mathcal
E$, and we say that the extension $\mathcal E$ {\it induces} $\psi$.
The set of equivalence classes of equivariant extensions of $B$ by
$Q$ of type $B\ri D$ inducing $\psi:Q\rightarrow \text{Coker}d$ is
denoted by
$$\Ext^\Gamma_{B\ri D}(Q,B,\psi).$$

Now, in order to study this set we apply the obstruction theory to
$\Gamma$-graded monoidal functors between strict $\Gamma$-graded
categorical groups $_\Gamma \Dis Q$ and $\mathbb P_{B\ri D}$, where
the {\it discrete} $\Gamma$-graded categorical group $_\Gamma \Dis
Q$ is defined by
$$_\Gamma \Dis Q =\int_\Gamma(Q, 0, 0).$$
This is just the strict $\Gamma$-graded categorical group associated
to the $\Gamma$-crossed module $(0,Q,0,0)$ (see Section 3). Thus,
the objects of $_\Gamma \Dis Q$ are the elements of $Q$  and its
morphisms   $ \sigma: x\rightarrow y$
  are the elements $\sigma \in\Gamma $ with $\sigma x=y$. Composition of morphisms is multiplication in   $\Gamma$.   The
graded tensor product is given by
\begin{align*} (x\xrightarrow{\sigma }y )\otimes (x'\xrightarrow{\sigma } y')=(xx'\xrightarrow{\sigma }yy' ).\end{align*}

We first prove the following lemma.

\begin{lem}\label{mrtc}
Let $B\ri D$ be a $\Gamma$-crossed module, and let $\psi: Q\ri
\Coker d$ be a $\Gamma$-homomorphism. For each $\Gamma$-graded
monoidal functor $(F,\widetilde{F}):\;_\Gamma\Dis Q\ri \mathbb
P_{B\ri D}$ which satisfies $F(1)=1$ and induces a pair of
$\Gamma$-homomorphisms $(\psi,0):(Q,0)\ri (\Coker d,\Ker d)$,  there
exists an equivariant group extension $\mathcal E_F$ of $B$ by $Q$
of type $B\ri D$ inducing $\psi$.
\end{lem}
The extension $\mathcal E_F$ is called an {\it  equivariant crossed
product extension associated} to the $\Gamma$-graded monoidal
functor $F$.

\begin{proof}
Let $(F,\widetilde{F}):\; _\Gamma\Dis Q\ri \mathbb P$  be a
$\Gamma$-graded monoidal functor. By \eqref{f}, it defines a
function $f: Q\times Q \cup ( Q\times \Gamma )\rightarrow B$ which
is normalized in the sense that
\begin{align}\label{eq54}f(x,1_\Gamma)=0=f(x,1) =f(1,y).\end{align}
The first equality holds since $F$ preserves   identities, the rest
ones hold since $F(1)=1$ and $F$ is compatible with the unit
constraints.

It follows from the definition of  morphism in $\mathbb P$
 that
\begin{equation} \sigma F(x)=df(x,\sigma)F(\sigma x),\label{mt1} \end{equation}
\begin{equation} F(x)F(y)=df(x,y)F(xy).\label{mt2} \end{equation}
According to the proof of Lemma \ref{t1}, the function $f$ satisfies
the equations \eqref{1s}-\eqref{3s}, but  it here takes values in
  $B$  instead of $\Ker d'$.


$\bullet$ Construction of the  crossed product $E_0=B\times_f Q$.

 The   $\Gamma$-group structure of $E_0$ is given by the rules
$$(b,x)+(c,y)= (b+\vartheta_{Fx}(c) + f(x,y),xy),$$
$$\sigma (b,x)=(\sigma b+f(x,\sigma),\sigma x).$$
Thanks to the conditions (\ref{1s}), (\ref{eq54}) and (\ref{mt2}),
  $B\times_f Q$ is actually a group. The zero is $(0,1)$ and $-(b,x)=(b',x^{-1})$, where $\vartheta_{Fx}(b')=-b-f(x,x^{-1})$.
  Moreover, $E_0$ is a $\Gamma$-group owning to the conditions   \eqref{2s}, \eqref{3s} and (\ref{mt1}).

Then, the following sequence is exact
 $$\mathcal E_F:\;\ 0\ri B\stackrel{j_0}{\rightarrow}E_0\stackrel{p_0}{\rightarrow}Q\ri 1,$$
where $ j_0(b)=(b,1),  p_0(b,x)= x,  b\in B, x\in Q.$ Since $j_0(B)$
is a normal subgroup in $E_0$,  $j_0:B\ri E_0$ is a $\Gamma$-crossed
module in which the action  $\vartheta^0:E_0\ri \Aut B$ is given by
conjugation.

$\bullet$ Embedding $\mathcal E_F$ into the diagram (\ref{dccs}).

We first  define a $\Gamma$-homomorphism $\varepsilon:E_0\ri D$.
Since  $(F,\widetilde{F})$ induces a $\Gamma$-homomorphism
$\psi:Q\ri \Coker d$ by $\psi(x)=[Fx]\in\Coker d$, the elements $Fx$
are representatives of Coker$d$ in $D$. For $(b,x)\in E_0$, we set
\begin{equation}\label{ep}
\varepsilon(b,x)=db.Fx.
\end{equation}
Then, $\varepsilon$ is a $\Gamma$-homomorphism thanks to the
conditions   (\ref{mt1}) and (\ref{mt2}).

It is easy to see that $\varepsilon\circ j_0=d$. Besides, for all
$(b,x)\in E_0,c\in B$,  one has
$\vartheta_{(b,x)}^0(c)=\vartheta_{\varepsilon(b,x)}(c)$, as
calculated below:
$$\vartheta^0_{(b,x)}(c)=j^{-1}_0[\mu_{(b,x)}(c,1)]=\mu_b[\vartheta_{Fx}(c)],$$
$$\vartheta_{\varepsilon(b,x)}(c)=\vartheta_{db.Fx}(c)=\mu_b[\vartheta_{Fx}(c)].$$
Thus,  $\mathcal E_F$ is embedded into the diagram (\ref{dccs}).

Finally, for all $x\in Q$,
$$q\varepsilon(b,x)=q(db.Fx)= q(Fx)= \psi(x)=\psi p_0(b,x),$$
so that the extension  $\mathcal E_F$ induces the
$\Gamma$-homomorphism  $\psi: Q\rightarrow \text{Coker}\ d$.
\end{proof}
Under the hypothesis of Lemma \ref{mrtc}, we state the following
theorem.
\begin{thm} [Schreier  theory for equivariant group extensions of the type of a $\Gamma$-crossed
module]\label{schr}  There is a bijection
$$
\Omega:\mathrm{Hom}_{(\psi,0)}[_\Gamma\mathrm{Dis}Q,\mathbb
P]\rightarrow \mathrm{Ext}_{B\ri D}^\Gamma(Q, B,\psi).$$
\end{thm}
\begin{proof}

{\it Step 1:  The $\Gamma$-graded monoidal functors
$F,F':\;_\Gamma\Dis Q\ri\mathbb P $ are homotopic if and only if the
corresponding associated equivariant extensions $\mathcal E_F,
\mathcal E_{F'}$ are equivalent.}

We first recall that every graded monoidal functor
$(F,\widetilde{F})$ is homotopic to one $(G,\widetilde{G})$ in which
$G(1)=1$. Hence, we can restrict our attention to this kind of
graded monoidal functors.

 Let $F, F': {_\Gamma\Dis Q}\ri\mathbb P$ be homotopic
by a homotopy $\alpha:F\ri F'$. Then, there exists a function
 $g:Q\ri B$ such that $\alpha_x=(g(x),1)$, that is,
\begin{equation}\label{dkmt}
Fx=dg(x)F'x.
\end{equation}
The naturality of $\alpha$ gives
\begin{equation}\label{ttn}
f(x,\sigma)+g(\sigma x)=\sigma g(x)+f'(x,\sigma).
\end{equation}
The coherence condition (\ref{3.4}) of the homotopy $\alpha$ implies
$g(1)=0$ and
\begin{equation}\label{dkk}
f(x,y)+g(xy)=g(x)+\vartheta_{F'x}g(y)+f'(x,y).
\end{equation}

By Lemma \ref{mrtc}, there exist extensions $\mathcal E_F$ and $
\mathcal E_{F'}$ associated to $F$ and $ F'$, respectively.  We
write
\begin{align*} \alpha^\ast: E_F&\rightarrow E_{F'}\\
(b,x)&\mapsto (b+g(x),x) \end{align*}

Then, thanks to the equations (\ref{ttn}) and (\ref{dkk}),
$\alpha^\ast$ is a   $\Gamma$-homomorphism. Further, the following
diagram commutes
\begin{align*}  \begin{diagram}
\xymatrix{ 0 \ar[r]& B \ar[r]^{j_0} \ar@{=}[d] &E_F  \ar[r]^{p_0}
\ar[d]^{\alpha^\ast} & Q \ar[r]\ar@{=}[d] & 1,&\;\;\; E_F
\ar[r]^\varepsilon&D \\
 0 \ar[r]& B \ar[r]^{j_0'}  &E_{F'}  \ar[r]^{p_0'}   & Q \ar[r]& 1,&\;\;\; E_{F'}
\ar[r]^{\varepsilon'}&D}
\end{diagram}
\end{align*}
and hence $\alpha^\ast$ is an isomorphism. It remains to show that
$\varepsilon'\alpha^\ast=\varepsilon$. It follows from the equations
\eqref{ep} and \eqref{dkmt} that
\[\begin{aligned}\varepsilon'\alpha^\ast(b,x)= &\varepsilon'(b+g(x),x)=d(b+g(x))F'x\\
= &d(b)d(g(x))F'x=d(b)Fx=\varepsilon(b,x).\end{aligned}\] Thus, two
extensions $\mathcal E_F$ and $\mathcal E_{F'}$ are equivalent.

Conversely, if $\alpha^*:E_F\rightarrow E_{F'}$ is an isomorphism,
 then $$\alpha^\ast(b,x)=(b+g(x),x),$$ where $g:Q\ri B$
is a function with $g(1)=0$. Thus, $\alpha_x=(g(x),1)$ is a homotopy
of $F$ and $F'$ as we see by retracing our steps.

{\it Step 2:  $\Omega $ is surjective.}

Assume that $\mathcal E$ is an equivariant extension of $B$ by $Q$
of type $B\rightarrow D$ inducing $\psi: Q\rightarrow \Coker d$ as
in the commutative diagram  (\ref{dccs}). We prove that $\mathcal E$
is equivalent to an extension $\mathcal E_F$  associated to some
$\Gamma$-graded monoidal functor
$(F,\widetilde{F}):\;_\Gamma\mathrm{Dis} Q\ri \mathbb P_{B\ri D}$.

For each  $x\in Q$, choose a representative $u_x\in E$ such that
$p(u_x)=x,\;u_1=0$. An element in  $E$ can be uniquely written  as
$b+u_x $, for $b\in B, x\in Q$.  The representatives
$\left\{u_x\right\}$ induce a normalized function   $f:Q\times Q\cup
Q\times\Gamma\ri B$ by
\begin{equation}    u_x+u_y=f(x,y)+u_{xy},\label{eq3'}\end{equation}
\begin{equation}
\sigma u_x=f(x,\sigma )+u_{\sigma x}.\label{eq4'} \end{equation} and
the automorphisms   $\varphi_x$  of $B$ by
$$\varphi_x=\mu_{u_x}:b\mapsto u_x+b-u_x.$$
It follows from the condition
 $H_2$ of the homomorphism $(id,\varepsilon)$ of
$\Gamma$-crossed modules that
\begin{equation*}    \vartheta_{\varepsilon u_x}=\mu_{ u_x}=\varphi_x. \label{}\end{equation*}

Then, the $\Gamma$-group structure of  $E$ can be described by
\begin{equation*} (b+u_x)+(c+u_y)= b+\varphi_x(c)+f(x,y)+u_{xy},
   \end{equation*}
\begin{equation*}
\sigma (b+u_x)=\sigma b +f(x,\sigma )+u_{\sigma x}.
  \end{equation*}

Since $\psi(x)=\psi p(u_x)=q\varepsilon(u_x)$, $\varepsilon(u_x)$ is
a representative of $\psi(x)$ in $D$. Thus,  we define a
$\Gamma$-graded monoidal functor
  $(F,\widetilde{F}):\;
\Dis_\Gamma Q\ri \mathbb P$ as follows.
  $$Fx=\varepsilon(u_x),\ F(x\xrightarrow{\sigma}\sigma x )=(f(x,\sigma),\sigma),\  \widetilde{F}_{x,y}=(f(x,y),1).$$
The equations  \eqref{eq4'} and  \eqref{eq3'}  show that
   $F(\sigma )$ and $\widetilde{F}_{x,y}$ are actually morphisms
   in
$\mathbb P$, respectively. The normality of the function
 $f(x,\sigma)$ gives
$F(id_x)= id_{Fx}$. Clearly, $F(1)=1.$ This together with the
normality of the function $f(x,y)$ imply the compatibility of
 $(F,\widetilde{F})$  with the unit constraints. The associativity law and the  $\Gamma$-group properties of $B$
imply the equations  \eqref{1s} - \eqref{3s}, respectively, in which
$\vartheta_{Fx}$ is replaced by
 $\varphi_x$. These equations show that
 $(F, \widetilde{F})$ is compatible with the associativity
 constraints,
 $\widetilde{F}_{x,y}$ is a natural isomorphism and
  $F$ preserves the composition of morphisms, respectively.

Finally, it is easy to check that the equivariant crossed product
extension $\mathcal E_F$ associated to
 $(F,\widetilde{F})$ is equivalent to the extension
 $\mathcal E$ by the $\Gamma$-isomorphism
$\alpha:(b,x)\mapsto b+u_x$.
\end{proof}

Moreover, each equivariant group extension of $B$ by
 $Q$  studied in \cite{C2002} may be viewed as an equivariant group extension
 of  type of
$\Gamma$-crossed module $(B,\Aut B,\mu,0)$. 
 Then,
$\mathbb P_{B\ri \Aut B}$ is just the {\it holomorph}
$\Gamma$-graded categorical group of a $\Gamma$-group $B$,
$\mathrm{Hol}_\Gamma B$.

\begin{hq} [Theorem 4.2 \cite{C2002}]\label{hq10'}  For $\Gamma$-groups $B$ and $Q$,
there exists a bijection
$$\mathrm{Hom}_\Gamma[\mathrm{Dis}_\Gamma Q,\mathrm{Hol}_\Gamma B]\rightarrow \mathrm{Ext}_\Gamma(Q, B).$$
\end{hq}

 Let $\mathbb P= \mathbb P_{B\rightarrow D}$
be the  $\Gamma$-graded categorical group associated to the
$\Gamma$-crossed module   $B\rightarrow D$. Since $\pi_0\mathbb
P=\Coker d$ and $\pi_1\mathbb P=\Ker d$, the reduced graded
categorical group  of $ {\mathbb P}$ is
$$\mathcal S_{\mathbb P}= \int_\Gamma(\mathrm{Coker} d, \mathrm{Ker} d,h), \; h\in Z^3_\Gamma (\mathrm{Coker} d, \mathrm{Ker} d).$$
Then, by \eqref{ct},   $\Gamma$-homomorphism  $\psi:Q\rightarrow \mathrm{Coker} d$
induces an {\it obstruction} $$\psi^\ast h\in Z^3_\Gamma (Q,
\mathrm{Ker} d).$$

Under this notion of obstruction, we state the following theorem.

\begin{thm}\label{dlc} Let $(B,D,d,\vartheta)$ be a $\Gamma$-crossed module, and
let $\psi: Q\rightarrow \mathrm{Coker} d$ be a
$\Gamma$-homomorphism.
 Then, the vanishing of $ \overline{\psi^*h}$ in
$H^3_\Gamma (Q, \mathrm{Ker} d)$ is  necessary and sufficient for
there to exist an equivariant extension of $B$ by $Q$ of  type
$B\rightarrow D$ inducing $\psi$. Further, if $\overline{\psi^*h} $
vanishes, then the equivalence classes of such extensions are
bijective with $H^2_\Gamma(Q,\mathrm{Ker}d)$.
\end{thm}

\begin{proof}  By the assumption,
$\overline{\psi^\ast h}=0$,  thus by Proposition \ref{dl2.2a}, there
exists
a $\Gamma$-graded monoidal functor $(\Psi,\widetilde{\Psi}):\;_\Gamma \Dis Q\ri \mathcal S_{\mathbb P}$. 
 Then, composition of $(\Psi, \widetilde{\Psi})$ and
$(H,\widetilde{H}): \mathcal S_{\mathbb P}\ri \mathbb P$ is a
$\Gamma$-graded monoidal functor $(F,\widetilde{F}):\;_\Gamma\Dis
Q\ri \mathbb P$. It is easy to see that $F$ induces the pair of
$\Gamma$-homomorphisms $(\psi,0)$, hence by Lemma \ref{mrtc}, we
obtain an associated extension $\mathcal E_F$.

Conversely, suppose that there is an equivariant extension as in the
diagram (\ref{dccs}). Let $\mathbb P'$ be the $\Gamma$-graded
categorical group associated to the $\Gamma$-crossed module $B\ri
E$.
 By Lemma \ref{t1}, there is a
$\Gamma$-graded monoidal functor $F:\mathbb P'\ri\mathbb P.$ Since
the reduced graded categorical group of $\mathbb P'$ is $_\Gamma
\Dis Q$,  $F$ induces a $\Gamma$-graded monoidal functor  of type
$(\psi,0)$ from $_\Gamma \Dis Q$ to $\mathcal S_{\mathbb
P}=\int_\Gamma(\Coker d,\Ker d,h)$.
 Now, by Proposition \ref{dl2.2a},
the obstruction of the pair $(\psi,0)$   vanishes  in
  $H^3_\Gamma(Q, \Ker d),$ that is,
$\overline{\psi^\ast h}=0$.

The final assertion  of the theorem follows from Proposition
\ref{dl2.2a} and Theorem  \ref{schr}.
\end{proof}

Note that if  $\Gamma={\bf 1}$, the trivial group, then  the set
$\Ext^{\bf 1}_{B\ri D}(Q,B,\psi)$ is just the set of equivalence
classes of group extensions of the type of a crossed module studied
in \cite{Br94,   Ded, Tay}. Thus, we obtain the following
consequence.

\begin{hq} [Theorem 5.2 \cite{Br94}]  Let $(B,D,d,\vartheta)$ be a crossed module, and
let $\psi: Q\rightarrow \mathrm{Coker} d$ be a group homomorphism.
 Then, there exists a 3-dimensional cohomology class $k(B,D,\psi)\in
 H^3(Q,\mathrm{Ker}d)$, called the obstruction,
 whose  vanishing  is  necessary and sufficient for there to exist
an extension of $B$ by $Q$ of  type $B\rightarrow D$ inducing
$\psi$. Further, if $k(B,D,\psi)$ vanishes, then the equivalence
classes of such extensions are bijective with
$H^2(Q,\mathrm{Ker}d)$.
\end{hq}

For the $\Gamma$-crossed module $(B,\Aut B,\mu,0)$, since $\Coker
\mu=\Out B,\Ker\mu=Z(B)$,  Theorem \ref{dlc} contains Theorem 4.1
\cite{C2002}.
\begin{hq} [Theorem 4.1 \cite{C2002}]\label{hq10'} Let $B,D $ be $\Gamma$-groups
and let $\psi: Q\rightarrow \mathrm{Out}B$ be a
$\Gamma$-homomorphism.
 Then, there exists the obstruction $\Obs(\psi)\in H^3_\Gamma  (Q, Z(B))$
 whose  vanishing  is  necessary and sufficient for there to exist
an equivariant extension of $B$ by $Q$ inducing $\psi$. Further, if
$\Obs(\psi)$ vanishes, then the equivalence classes of such
extensions are bijective with $H^2_\Gamma (Q,Z(B))$.
\end{hq}

\begin{center}
{}
\end{center}

\end{document}